\documentclass[11pt]{article}
\usepackage{a4wide}

\usepackage{amsmath,amssymb,amsthm,caption}
\usepackage[mathscr]{eucal}
\usepackage{graphics}
\usepackage[hypertex]{hyperref}
\usepackage[usenames]{color}

\usepackage{epsfig}
\usepackage{psfrag}

\definecolor{red}{rgb}{1,0.00,0.00}
\definecolor{blue}{rgb}{0,0.08,0.55}
\definecolor{green}{rgb}{0.0,0.6,0}

\author{Ievgen Bondarenko, Tullio Ceccherini-Silberstein,\\ Alfredo Donno, Volodymyr Nekrashevych}
\title{\textbf{On a family of Schreier graphs of intermediate growth associated with a self-similar group}}

\newcommand{\nucl}{\mathcal{N}}

\newcommand{\Cof}{\textrm{Cof}\,}
\newcommand{\gr}{\Gamma}

\newcommand{\Sym}{\textrm{Sym}}

\newcommand{\Aut}{\mathop{\rm Aut}\nolimits}

\newtheorem{thm}{Theorem}
\newtheorem{prop}[thm]{Proposition}
\newtheorem{cor}[thm]{Corollary}
\newtheorem{lemma}{Lemma}
\theoremstyle{definition}

\begin{document}
\maketitle

\begin{abstract}
For every infinite sequence $\omega=x_1x_2\ldots$, with
$x_i\in\{0,1\}$, we construct an infinite $4$-regular graph $X_{\omega}$.
These graphs are precisely the Schreier graphs of the action of a certain
self-similar group on the space $\{0,1\}^{\infty}$. We solve the
isomorphism and local isomorphism problems for these graphs, and
determine their automorphism groups.
Finally, we prove that all graphs $X_\omega$ have intermediate growth.\\


\end{abstract}


\section{Introduction}
In  \cite{CFS}, F. Fiorenzi, F. Scarabotti and the second named
author of the present paper gave an exposition of some results of
M. Gromov from \cite{gromov} on symbolic dynamics on infinite
graphs. Let $X = (V,E)$ be a graph of uniformly bounded valence.
Let $r \geq 0$, $u \in V$ and denote by $B_r(u) = \{w \in V:
d(u,w) \leq r\}$ the \emph{ball of radius r centered at u}, where
$d \colon V \times V \to \mathbb{R}^{+}$ is the geodesic distance
on $X$. Given two vertices $u,v\in V$, one says that $u \sim_r v$
if there exists a graph isomorphism $\psi_r(u,v) \colon B_r(u) \to
B_r(v)$ such that $\psi_r(u)= v$. It is immediate that $\sim_r$ is
an equivalence relation on $V$; the corresponding equivalence
classes are called \emph{r-classes}. Since $X$ has uniformly
bounded valence, for each integer $r \geq 0$ there exist only
finitely many $r$-classes for the vertices of $X$. The collection
$P(X)$ of all graph isomorphisms of the form $\psi_r(u,v)$, where
$u\sim_r v$, $r \geq 0$, constitutes the \emph{pseudogroup} of
\emph{partial isometries} of the graph $X$. Gromov called such a
pseudogroup \emph{of dense holonomy} provided that for every $r >
0$ there exists a $D_r > 0$ such that for every $u,v \in V$ there
exists $w \in V$ such that $d(u,w) < D_r$ and $v \sim_r w$. For
example, if $G$ is a finitely generated group, $S \subset G$ is a
finite symmetric generating subset, then the Cayley graph $X =
\mathcal{C}(G,S)$ of $G$ with respect to $S$ is a regular graph of
degree $\vert S \vert$ (and thus of uniformly bounded valence).
Moreover, for all $u,v \in V$, the left multiplication by $h =
vu^{-1}$ yields a graph isomorphism $B_r(u) \to B_r(v)$ such that
$u \mapsto v$. It follows that for each $r \geq 0$ there exists a
unique $r$-class in $X$. In fact $G \hookrightarrow \Aut(X)$ and
this constitutes a trivial example of such dense holonomy
pseudogroups of isometries of a (bounded valence) graph. A. \.Zuk
asked for non-trivial examples of such graphs $X$, possibly with a
trivial automorphism group $\Aut(X)$, and in \cite[Example
3.23]{CFS} an explicit example is provided. Consider the Cayley
graph $X = \mathcal{C}(\mathbb{Z}, \{\pm 1\}) = (V,E)$ and add new
edges according to the following recursive rule. We first connect
all pairs of vertices of the form $(2n,2n+2)$, where $n \in
\mathbb{Z}$. After this step, all vertices corresponding to even
integers have degree $4$. Consider now the remaining vertices of
degree $2$: these are the odd vertices. Note that $1$ and $-1$ are
the two vertices of degree $2$ which are closest to $0$. We choose
the vertex $1$ and we then connect all pairs of vertices of the
form $(2n+1,2n+5)$, where $n \in \mathbb{Z}$. This way, also the
odd vertices which are congruent to $1$ mod $4$ now have degree
$4$. Observe that the vertices that still have degree $2$ are the
odd vertices which are congruent to $3$ mod $4$: in particular,
the vertex of degree $2$ which is closest to $0$ is $-1$. We then
connect all pairs of vertices of the form $(2n-1,2n+7)$, where $n
\in \mathbb{Z}$. This way, also the odd vertices which are
congruent to $-1$ mod $8$ now have degree $4$. And so on (see
Section \ref{section2} for more details). The resulting graph $X'$
(which in the present paper is denoted by $X_{(10)^\infty}$) is
regular of degree $4$, its pseudogroup of partial isometries has
dense holonomy and, moreover, $\Aut(X')$ is trivial. This last
result is easily deduced from the following fact. As one easily
checks, the graph $X'$ is \emph{generic}, in the sense that for
all $u,v\in V(X')$ there exists $r>0$ such that $u\not\sim_r v$.
As genericity is equivalent to the triviality of the automorphism
group of the graph, this gives our claim. In \cite{CFS}, it is
also shown that $X'$ is an amenable graph and it is observed,
after a remark of the last named author of the present paper, that
$X'$ is the Schreier graph associated with the action of a
self-similar group on the boundary of the rooted binary tree.

In \cite{omega_periodic}, I. Benjamini and C. Hoffman considered a
family of amenable graphs, called {\it $\omega$-periodic graphs},
whose construction is similar to that of $X'$. In particular,
their \lq\lq basic example\rq\rq corresponds to the graph that in
the present paper is denoted by $X_{0^\infty}$. They proved that
this graph has intermediate growth, that is, it is superpolynomial
and subexponential and they also proved, after a remark of L.
Bartholdi, that this graph is an example of Schreier graph. They
also provided examples within the family of $\omega$-periodic
graphs having polynomial (resp. exponential) growth.

In the present paper, with each right-infinite sequence $\omega =
x_1x_2\ldots \in \{0,1\}^\infty$, we associate an infinite
$4$-regular graph $X_\omega$. Moreover, after saying that two
sequences $\omega = x_1x_2\ldots$ and $\omega'=y_1y_2\ldots$ are
{\it cofinal} (resp. {\it anticofinal}) provided that there exists
$i_0$ such that $x_i=y_i$ (resp. $x_i = 1-y_i$), for every $i\geq
i_0$, we prove the following results:
\begin{itemize}
\item $X_\omega \cong X_{\omega'}$ if and only if $\omega$ and
$\omega'$ are either cofinal or anticofinal (Theorem
\ref{thm_IsomProblem});
\item $\Aut(X_{\omega})$ is trivial if
$\omega$ is neither cofinal nor anticofinal to $0^\infty$, and
$\Aut(X_\omega) = \mathbb{Z}/2\mathbb{Z}$ otherwise (Corollary
\ref{cor_AutomorphisGroup});
\item $X_{\omega}$ and $X_{\omega'}$ are locally isomorphic if and only if either
both $\omega$ and $\omega'$ are cofinal or anticofinal with
$0^{\infty}$, or both $\omega$ and $\omega'$ are neither cofinal
nor anticofinal with $0^{\infty}$ (Theorem \ref{thm_LocalIsom});
\item $X_{\omega}$ has dense holonomy and is generic if and only if
$\omega$ is neither cofinal nor anticofinal with $0^\infty$
(Theorem \ref{thmholonomy});
\item for each $\omega$, the graph $X_\omega$ is isomorphic to the
orbital Schreier graph $\Gamma_\omega$ of the word $\omega$ under
the action of a self-similar group $G$ (Theorem
\ref{thmschreierisom});
\item $X_\omega$ has intermediate growth, and therefore is amenable, for all $\omega$
(Theorem \ref{thmintermediate}).
\end{itemize}

In the Appendix, a detailed study of finite Gelfand pairs
associated with the action of the group $G$ on each level of the
rooted binary tree is presented. This leads in particular to a
description of the decomposition of the corresponding permutations
representation into irreducible submodules and to an explicit
expression for the associated spherical functions. The key step is
to prove that the action of $G$ on each level of the tree is
$2$-point homogeneous. Incidentally, this automatically gives the
symmetry of these Gelfand pairs.

\section{Definition of the graphs $X_{\omega}$ and the isomorphism
problem}\label{section2}

Consider the binary alphabet $\{0,1\}$, and let $\{0,1\}^{\infty}$
be the set of all (right-)infinite sequences $x_1x_2\ldots$ with
$x_i\in\{0,1\}$. We associate an infinite 4-regular graph
$X_{\omega}=(V_{\omega},E_{\omega})$ with each sequence $\omega\in
\{0,1\}^{\infty}$. The vertex set $V_{\omega}$ of every graph
$X_\omega$ is the set $\mathbb{Z}$ of integer numbers. The edge
set $E_{\omega}$ depends on the sequence $\omega=x_1x_2\ldots$ and
is defined as follows. For every $n\geq 1$, we set
\begin{eqnarray}\label{eq_defi_an}
a_n^{\omega} = x_1+x_2 2+ x_3 2^2+\cdots+x_{n-1}2^{n-2}-\overline{x}_n2^{n-1}=\sum_{i=1}^nx_i 2^{i-1}-2^{n-1},
\end{eqnarray}
where $\overline{x}= 1 - x$, for each $x\in\{0,1\}$. Notice that
if $\omega=0^{\infty}$ (resp. $\omega=1^{\infty}$) one has
$a_n^{\omega}=-2^{n-1}$ (resp. $a_n^{\omega}=2^{n-1}-1$) for all
$n\geq 1$. In the general case, we get the inequalities
$$
-2^{n-1}\leq a_n^{\omega}\leq 2^{n-1}-1,
$$
for every $n\geq 1$. Put $E^0_{\omega} = \{(z,z+1): z\in
\mathbb{Z}\}$ and, for every $n\geq 1$, define
\[
E^n_{\omega} = \{(2^nz-a_n^{\omega},2^n(z+1)-a_n^{\omega}): z\in
\mathbb{Z}\}.
\]
Then the edge set $E_{\omega}$ of the graph $X_{\omega}$ is given
by the disjoint union $\coprod_{n=0}^{\infty}E^n_{\omega}$, with
possibly a loop rooted at the (unique) vertex which is not
incident to any edge of $\coprod_{n=1}^{\infty}E^n_{\omega}$.

The graph $X_{\omega}$ can be constructed step by step by reading
consequently the letters of the binary sequence $\omega$ and
adding the edges from the set $E^n_{\omega}$. Denote by
$X^0_{\omega}$ the graph with the vertex set $\mathbb{Z}$ and the
edge set $E^0_{\omega}$ and, for each $n\geq 1$, let
$X^n_{\omega}$ be the graph with the vertex set $\mathbb{Z}$ and
the edge set $\coprod_{k=0}^nE^k_{\omega}$. Note that all vertices
of the graph $X^0_{\omega}$ have degree $2$, while the graph
$X^n_{\omega}$, for $n\geq 1$, contains also vertices of degree
$4$. Suppose we have constructed the graph $X^{n-1}_{\omega}$ and
we read the $n$-th letter $x_n$ of the sequence $\omega$. If
$x_n=0$, then we find the smallest positive integer which has
degree $2$, viewed as a vertex of the graph $X^{n-1}_{\omega}$
(the first vertex of degree $2$ to the right from zero). If
$x_n=1$, then we find the largest nonpositive integer which has
degree $2$ as a vertex of $X^{n-1}_{\omega}$ (the first vertex of
degree $2$ to the left from zero). In both cases this integer
number is precisely $-a_n^{\omega}$, see Figure \ref{FIG.1}. Then
the graph $X^n_{\omega}$ is obtained from the graph
$X^{n-1}_{\omega}$ by connecting all second consecutive vertices
which are congruent to $-a^{\omega}_n$ modulo $2^n$ (they are of
degree $2$ in $X^{n-1}_{\omega}$). These vertices become the new
vertices of degree $4$ in the graph $X^n_{\omega}$ and the
corresponding new edges constitute the set $E_{\omega}^n$. In
particular, the vertex of $X_{\omega}$ corresponding to the
integer $-a_n^{\omega}$ is the closest vertex to 0 which is
incident to an edge in $E^n_{\omega}$, in formulae,
\begin{eqnarray}\label{eq_a_n_minimon}
|a^{\omega}_n|= \min\left\{|z|: z\in \mathbb{Z}\setminus
\coprod_{i=1}^{n-1}(2^i\mathbb{Z}-a^{\omega}_i)\right\}.
\end{eqnarray}

\begin{figure}[h]
\begin{center}
\psfrag{0}{$0$} \psfrag{1}{$1$} \psfrag{00}{$00$}
\psfrag{01}{$01$} \psfrag{10}{$10$} \psfrag{11}{$11$}
\psfrag{zero}{$0$} \psfrag{zero2} {$-a^\omega_1$} \psfrag{zero4}
{$-a^\omega_2$} \psfrag{X0} {$X_{\omega}^0$}
\psfrag{X1}{$X_{\omega}^1$} \psfrag{X2}{$X_{\omega}^2$}
\psfrag{a}{$-a^\omega_1$} \psfrag{b}{$-a^\omega_2$}
\psfrag{c}{$-a^\omega_2$} \psfrag{d}{$-a^\omega_2$}
\psfrag{-2}{$-2$} \psfrag{-1}{$-1$}\psfrag{dot}{$\cdot$}
\psfrag{uno}{$1$}\psfrag{due}{$2$} \psfrag{w}{.} \psfrag{ww}{.}
\psfrag{www}{.} \psfrag{wwww}{.}
\epsfig{file=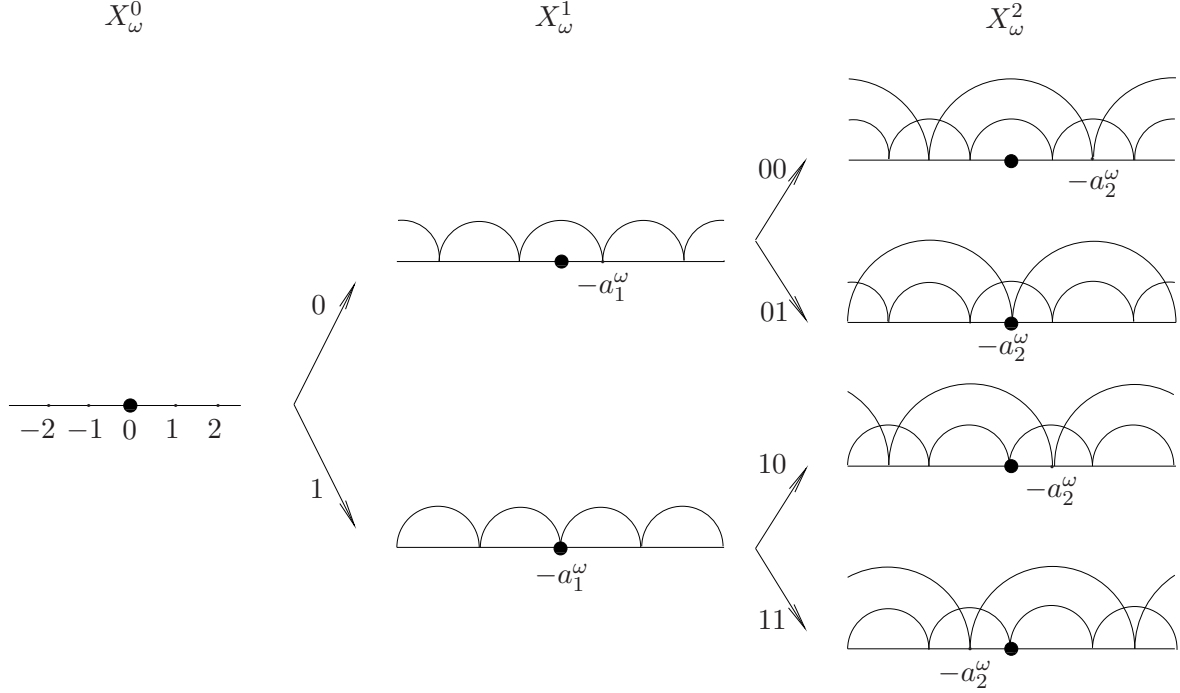,width=1\textwidth} \caption{The
construction of the graphs $X_\omega^0$, $X_\omega^1$ and
$X_\omega^2$.\label{FIG.1}}
\end{center}
\end{figure}


If the sequence $\omega$ contains both infinitely many $0$'s and
$1$'s, then each vertex of the graph $X_{\omega}$ is incident to
an edge of $\coprod_{n=1}^{\infty}E^n_{\omega}$ and so there is no
loop in $X_{\omega}$. Indeed, if there exists a vertex $v$ which
is not incident to any edge of
$\coprod_{n=1}^{\infty}E^n_{\omega}$ and this vertex corresponds
to a nonpositive (resp. positive) integer, then there exists
$n_0\geq 1$ such that $x_n=0$ (resp. $x_n=1$) for all $n\geq n_0$.
On the other hand, if $\omega = 0^{\infty}$, then $0$ is the
unique vertex of $X_{0^{\infty}}$ which is not incident to any
edge of $\coprod_{n=1}^{\infty}E^n_{\omega}$ and so there is a
loop at $0$. Similarly, if $\omega=1^{\infty}$, then $1$ is the
unique vertex of $X_{1^{\infty}}$ which is not incident to any
edge of $\coprod_{n=1}^{\infty}E^n_{\omega}$ and so there is a
loop at $1$. In the general case, if $\omega=x_1x_2\ldots
x_n0^{\infty}$ then the graph $X_{\omega}$ has a loop at the
vertex $-\sum_{i=1}^n x_i 2^{i-1}$; similarly, if
$\omega=x_1x_2\ldots x_n1^{\infty}$ then the graph has a loop at
$1-\sum_{i=1}^n (x_i-1) 2^{i-1}$.
\begin{figure}[h]
\begin{center}
\psfrag{0}{0}\psfrag{-1}{-1}\psfrag{-2}{-2}\psfrag{-3}{-3}\psfrag{-4}{-4}\psfrag{-5}{-5}\psfrag{-6}{-6}\psfrag{-7}{-7}
\psfrag{-8}{-8}\psfrag{-9}{-9}\psfrag{-10}{-10}\psfrag{-11}{-11}\psfrag{-12}{-12}\psfrag{-13}{-13}\psfrag{-14}{-14}
\psfrag{1}{1}\psfrag{2}{2}\psfrag{3}{3}\psfrag{4}{4}\psfrag{5}{5}\psfrag{6}{6}\psfrag{7}{7}
\psfrag{8}{8}\psfrag{9}{9}\psfrag{10}{10}\psfrag{11}{11}\psfrag{12}{12}\psfrag{13}{13}\psfrag{14}{14}\psfrag{bullet}{$\bullet$}\psfrag{dot}{$\cdot$}
\includegraphics[width=1\textwidth]{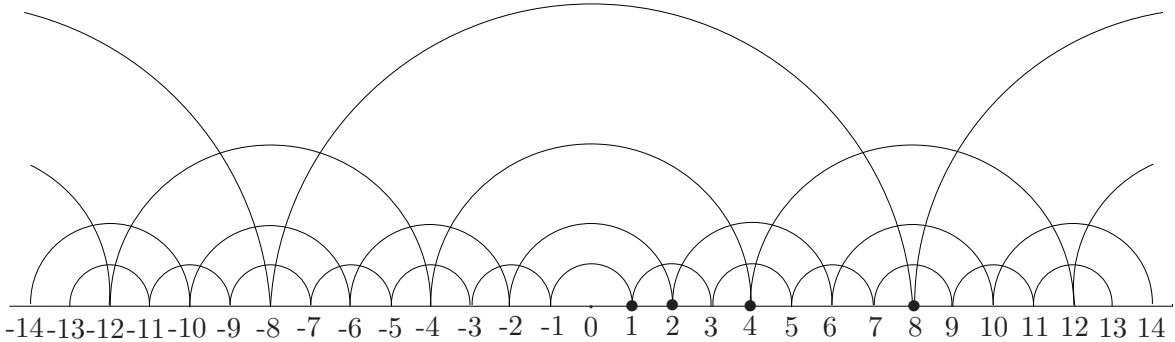} \caption{The graph $X_{0^{\infty}}$, with
$a_n^{0^{\infty}}=-2^{n-1}$.}
\end{center}
\end{figure}

\begin{figure}[h]
\begin{center}
\psfrag{0}{0}\psfrag{-1}{-1}\psfrag{-2}{-2}\psfrag{-3}{-3}\psfrag{-4}{-4}\psfrag{-5}{-5}\psfrag{-6}{-6}\psfrag{-7}{-7}
\psfrag{-8}{-8}\psfrag{-9}{-9}\psfrag{-10}{-10}\psfrag{-11}{-11}\psfrag{-12}{-12}\psfrag{-13}{-13}\psfrag{-14}{-14}
\psfrag{1}{1}\psfrag{2}{2}\psfrag{3}{3}\psfrag{4}{4}\psfrag{5}{5}\psfrag{6}{6}\psfrag{7}{7}\psfrag{bullet}{$\bullet$}
\psfrag{8}{8}\psfrag{9}{9}\psfrag{10}{10}\psfrag{11}{11}\psfrag{12}{12}\psfrag{13}{13}\psfrag{14}{14}
\includegraphics[width=1\textwidth]{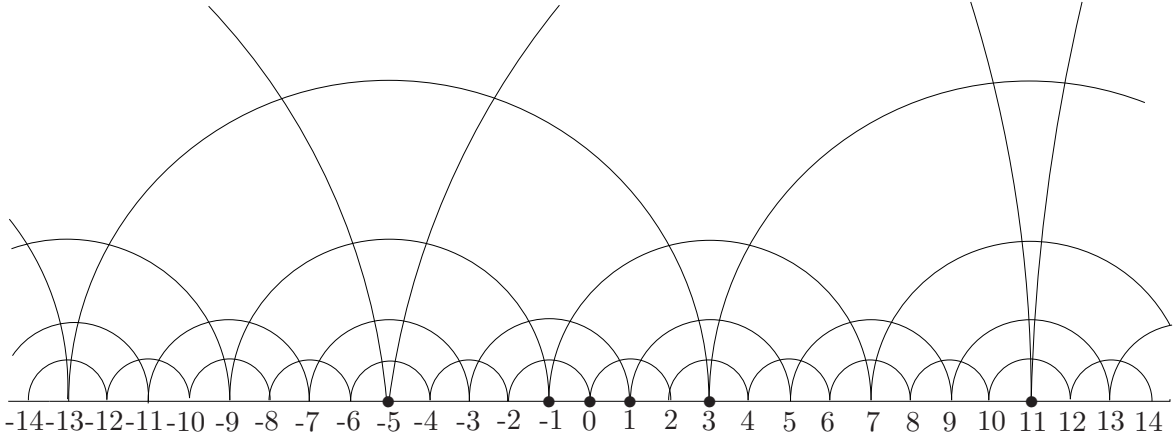} \caption{The graph $X_{(10)^{\infty}}$, with
$a_n^{(10)^{\infty}}=\frac{-1+(-2)^{n-1}}{3}$.}
\end{center}
\end{figure}


Let us consider the isomorphism problem for the family of graphs
$X_{\omega}$, with $\omega\in\{0,1\}^{\infty}$.

\begin{lemma}\label{lemma_JustOneLetter}
Let $\omega,\omega'\in \{0,1\}^{\infty}$ and suppose that $\omega$ and $\omega'$ differ only at the $n$-th letter, namely
\begin{align*}
\omega &= x_1x_2\ldots x _{n-1}x_nx_{n+1}\ldots\ldots\\
\omega'&= x_1x_2\ldots x _{n-1}\overline{x}_nx_{n+1}\ldots\ldots
\end{align*}
Then the map $\Phi_n:V_{\omega}\longrightarrow V_{\omega'}$
defined by the rule
\begin{eqnarray}\label{eq_map_Phi_oneX}
\Phi_n(z) = z - (\overline{x}_n-x_n)2^{n-1},
\end{eqnarray}
for all $z\in V_\omega = \mathbb{Z}$, is a graph isomorphism.
\end{lemma}
\begin{proof}
It suffices to observe that, under our assumptions,
$$
a_i^{\omega'}=a_i^{\omega} \qquad \mbox{for each } 1\leq i \leq n-1
$$
and
$$
a_i^{\omega'}=a_i^{\omega} + (\overline{x}_n-x_n)2^{n-1} \qquad
\mbox{for each } i \geq n,
$$
so that $\Phi_n$ preserves the adjacency relation between
vertices.
\end{proof}

\begin{lemma}\label{lemma_Total_AntiCofinality}
Let $\omega= x_1x_2\ldots\in \{0,1\}^{\infty}$ and put
$\overline{\omega} = \overline{x}_1\overline{x}_2\ldots$. Then the
map $\Psi:V_{\omega}\longrightarrow V_{\overline{\omega}}$ defined
by the rule
\[
\Psi(z) = -z +1,
\]
for all $z\in V_\omega = \mathbb{Z}$, is a graph isomorphism.
\end{lemma}
\begin{proof}
Since $x_i + \overline{x}_i=1$ for each $i\geq 1$, it follows from \eqref{eq_defi_an} that
$$
a_i^{\omega}+ a_i^{\overline{\omega}} = 1+2+2^2 +\cdots +2^{i-2}-2^{i-1}=-1
$$
for each $i\geq 1$ and so $\Psi$ preserves the adjacency relation
between vertices.
\end{proof}

Two sequences $\omega = x_1x_2\ldots$ and $\omega' = y_1y_2\ldots$
over the alphabet $\{0,1\}$ are said to be \textit{cofinal} (resp.
\textit{anticofinal}) if there exists $i_0$ such that $x_i=y_i$
(resp. $x_i=\overline{y}_i$), for every $i\geq i_0$. It is clear
that cofinality is an equivalence relation. The corresponding
equivalence classes are called the cofinality classes and we
denote by $Cof(\omega)$ the cofinality class of
$\omega\in\{0,1\}^{\infty}$.

\begin{thm}\label{thm_IsomProblem}
Two graphs $X_{\omega}$ and $X_{\omega'}$, with
$\omega,\omega'\in\{0,1\}^{\infty}$, are isomorphic if and only if
the sequences $\omega$ and $\omega'$ are either cofinal or
anticofinal.
\end{thm}
\begin{proof}
Let $\omega=x_1x_2\ldots$ and $\omega'=y_1y_2\ldots$ with
$x_i,y_i\in \{0,1\}$. First, consider the case when the sequences
$\omega$ and $\omega'$ are cofinal. Then $\omega$ and $\omega'$
differ only at finitely many indices, say $i_1,i_2, \ldots, i_k$,
i.e., $y_{i_j}= \overline{x}_{i_j}$ for $j=1,\ldots,k$, and $y_l =
x_l$ for $l\not\in \{i_1, \ldots, i_k\}$. Then, the composition
$\Phi = \Phi_{i_k}\circ \cdots \circ \Phi_{i_1}$ of the
isomorphisms defined in \eqref{eq_map_Phi_oneX} yields an
isomorphism between the graphs $X_{\omega}$ and $X_{\omega'}$. If
$\omega$ and $\omega'$ are anticofinal, then $\overline{\omega}$
and $\omega'$ are cofinal, and hence the graphs
$X_{\overline{\omega}}$ and $X_{\omega'}$ are isomorphic. Then the
graphs $X_{\omega}$ and $X_{\omega'}$ are isomorphic by
Lemma~\ref{lemma_Total_AntiCofinality}.


For the converse, suppose that the graphs $X_{\omega}$ and
$X_{\omega'}$ are isomorphic. If the sequence $\omega$ is cofinal
or anticofinal with $0^{\infty}$, then there exists a vertex with
a loop in the graph $X_{\omega}$. Hence the graph $X_{\omega'}$
must contain a vertex with a loop and thus $\omega'$ is cofinal or
anticofinal with $0^{\infty}$. So, in the sequel, we assume that
both $\omega$ and $\omega'$ have infinitely many $0$'s and $1$'s.

The following remark is crucial. The edges in the set $E^0_\omega$
have a unique property, namely that for two consecutive edges
$(z,z+1)$ and $(z+1,z+2)$, the vertices $z$ and $z+2$ are adjacent
in the graph $X_{\omega}$ (by an edge in $E_\omega^1$) if either
$z$ is even and $\omega=1x_2x_3\ldots$, or if $z$ is odd and
$\omega=0x_2x_3\ldots$. This property does not hold for the edges
in $E^i_\omega$ with $i\geq 1$, because the graph is $4$-regular
and each vertex is incident to exactly two edges in $E^0_\omega$
and to only other two edges in $E^i_\omega$. Then it is easy to
see that the edge set $E^0_\omega$ must be preserved under graph
isomorphisms. This implies that a graph isomorphism
$\varphi:V_{\omega}\longrightarrow V_{\omega'}$ is either of the
form $z\mapsto z+t$ or of the form $z\mapsto -z+t$, for a fixed
$t\in \mathbb{Z}$.

Suppose that $\varphi:V_{\omega}\longrightarrow V_{\omega'}$ is an
isomorphism of the form $\varphi(z)=z+t$, for all
$z\in\mathbb{Z}$. We will show that the sequence $\omega'$ can be
eventually recovered by using the property (\ref{eq_a_n_minimon}).
This property implies that the vertex
$\varphi(-a_n^{\omega})=-a_n^{\omega}+t$ of the graph
$X_{\omega'}$ is the closest one to the vertex $\varphi(0) = t$
that is incident to an edge of $E_{\omega'}^n$. We set
$$
I_t=\{n\in \mathbb{N} : 1-2^{n-1}\leq -a_n^{\omega} + t \leq 2^{n-1}\}.
$$
We claim that the set $I_t$ satisfies the following properties:
\begin{enumerate}
\item the set $I_t$ is nonempty;
\item if $n\in I_t$, then  $n+k\in I_t$ for all $k\geq 1$.
\end{enumerate}

\indent{\it Proof of 1.} Case $t>0$. There exists a unique integer $n_1 \geq 1$ such that $2^{n_1-1}\leq  t< 2^{n_1}$. We
choose the smallest $i\geq 1$ such that $x_{n_1+i}=1$ (the index $i$ is well defined since $\omega$ contains infinitely
many $1$'s). With this choice, one has
$$
a_{n_1+i+1}^{\omega}\leq 1+2+\cdots +2^{n_1+i-1}=2^{n_1+i}-1
$$
and so $-a_{n_1+i+1}^{\omega} \geq 1-2^{n_1+i}$. Then a fortiori $-a_{n_1+i+1}^{\omega}+t \geq -2^{n_1+i}+1$, because $t$
is positive. On the other hand, we also have $a_{n_1+i+1}^{\omega} \geq 2^{n_1+i-1}-2^{n_1+i}$ by construction and so
$-a_{n_1+i+1}^{\omega}\leq 2^{n_1+i}-2^{n_1+i-1}$. Hence,
$$
-a_{n_1+i+1}^{\omega}+t\leq 2^{n_1+i}-2^{n_1+i-1}+2^{n_1}\leq 2^{n_1+i},
$$
because $-2^{n_1+i-1}+2^{n_1}\leq 0$. This implies that $n_1+i+1\in I_t$.

\indent Case $t<0$. There exists a unique integer $n_1\geq 1$ such that $-2^{n_1}< t\leq -2^{n_1-1}$. We choose the
smallest $i\geq 1$ such that $x_{n_1+i}=0$ and $x_{n_1+i+1}=1$ (the index $i$ is well defined since $\omega$ has
infinitely many 0's and 1's). With this choice, one has
$$
a_{n_1+i+1}^{\omega}\leq 1+2+\cdots + 2^{n_1+i-2}=2^{n_1+i-1}-1
$$
and so $-a_{n_1+i+1}^{\omega} \geq -2^{n_1+i-1}+1$. This implies
$$
-a_{n_1+i+1}^{\omega}+t \geq -2^{n_1+i-1}+1-2^{n_1}\geq -2^{n_1+i}+1,
$$
because $-2^{n_1}\geq -2^{n_1+i-1}$. Moreover, we also have $a_{n_1+i+1}^{\omega} \geq -2^{n_1+i}$ by construction and so
$-a_{n_1+i+1}^{\omega}\leq 2^{n_1+i}$. Then a fortiori $-a_{n_1+i+1}^{\omega}+t \leq 2^{n_1+i}-1$, because $t \leq -1$.
Hence, $n_1+i+1\in I_t$.\\

\indent{\it Proof of 2.} Suppose that $n\in I_t$, i.e. $1-2^{n-1}\leq -a_n^{\omega}+t\leq 2^{n-1}$. It follows from
\eqref{eq_defi_an} that
\begin{eqnarray}\label{eq_inThIso_differencean}
a_{n+k}^{\omega}-a_n^{\omega} = 2^{n-1}+x_{n+1}2^n+\cdots +x_{n+k-1}2^{n+k-2}-\overline{x}_{n+k}2^{n+k-1}
\end{eqnarray}
and so
$$
a_{n+k}^{\omega}-a_n^{\omega} \leq 2^{n-1}+2^n+\cdots+2^{n+k-2} = 2^{n-1}(2^k-1) = 2^{n+k-1}-2^{n-1}.
$$
This gives
$$
-a_{n+k}^{\omega}+t\geq -a_n^{\omega}-2^{n+k-1}+2^{n-1}+t\geq -2^{n+k-1}+1.
$$
On the other hand, it follows from \eqref{eq_inThIso_differencean} that $a_{n+k}^{\omega}-a_n^{\omega}\geq
2^{n-1}-2^{n+k-1}$ and so
$$
-a_{n+k}^{\omega}+t\leq -a_n^{\omega}+t-2^{n-1}+2^{n+k-1}\leq 2^{n+k-1}.
$$
The properties $1$ and $2$ are proved.

Let $n_0 = n_0(t)$ be the least element of the set $I_t$. Then it
follows from the definition of $I_t$ and the property $2$ that
\begin{equation}\label{eq_inThIso_anEquality}
-a_n^{\omega'}= -a_n^{\omega}+t \qquad \mbox{for all }n\geq n_0.
\end{equation}
In other words, there exists $n_0$ depending on the translation
$t$ such that the closest integer to $\varphi(0)$ which is
incident to an edge of $E^n_{\omega'}$ coincides with the closest
integer to 0 which is incident to an edge of $E^n_{\omega'}$, for
all $n\geq n_0$. We will rewrite \eqref{eq_defi_an} in the form
\[
a_n^{\omega} = a_{n-1}^{\omega}+2^{n-2}(2x_n-1),\qquad\qquad a_n^{\omega'} = a_{n-1}^{\omega'}+2^{n-2}(2y_n-1).
\]
Then \eqref{eq_inThIso_anEquality} implies
\begin{eqnarray*}
0=-a_{n}^{\omega}+t+a_{n}^{\omega'} &=&
(-a_{n-1}^{\omega}+t+a_{n-1}^{\omega'})+2^{n-1}(y_{n}-x_{n})\\
&=&  2^{n-1}(y_{n}-x_{n}),
\end{eqnarray*}
for all $n>n_0$. Hence $x_n=y_n$ for all $n>n_0$, and therefore
the sequences $\omega$ and $\omega'$ are cofinal.

Suppose now that $\varphi:V_{\omega}\longrightarrow V_{\omega'}$
is an isomorphism of the form $\varphi(z)=-z+t$. Consider the
isomorphism $\Psi: V_{\omega'}\longrightarrow
V_{\overline{\omega'}}$ defined in
Lemma~\ref{lemma_Total_AntiCofinality}. Then the composition
$$
\Psi \circ \varphi : V_{\omega}\longrightarrow
V_{\overline{\omega'}}
$$
is an isomorphism of the form $z\mapsto z+t$ between the graphs $X_{\omega}$ and $X_{\overline{\omega'}}$. Then it
follows from the first part of the proof that $\omega$ and $\overline{\omega'}$ are cofinal, and this implies that
$\omega$ and $\omega'$ are anticofinal.
\end{proof}


\begin{cor}\label{cor_AutomorphisGroup}
If $\omega$ is either cofinal or anticofinal with $0^{\infty}$ then $Aut(X_\omega) \cong \mathbb{Z}/2\mathbb{Z}$,
otherwise the  automorphism group $Aut(X_{\omega})$ is trivial.
\end{cor}
\begin{proof}
Suppose first that the sequence $\omega=x_1x_2\ldots\in
\{0,1\}^{\infty}$ contains infinitely many $0$'s and $1$'s. Then
an automorphism $\varphi$ of the graph $X_{\omega}$ is necessarily
of the form $\varphi(z)=z+t$. It follows from
\eqref{eq_inThIso_anEquality}, with $\omega'=\omega$, that $t=0$
and therefore $\varphi=id_{X_\omega}$. Hence the group
$Aut(X_\omega)$ is trivial.

Now suppose that $\omega$ contains either finitely many $0$'s or $1$'s. In both cases, the sequence $\omega$ is either
cofinal or anticofinal with the infinite word $0^\infty$. Hence the graph $X_\omega$ is isomorphic to the graph
$X_{0^\infty}$ and so we are only left to show that $Aut(X_{0^\infty})\cong \mathbb{Z}/2\mathbb{Z}$. Take an automorphism
$\varphi\in Aut(X_{0^\infty})$. Since the vertex $0$ is the unique vertex with a loop, we get $\varphi(0)=0$. As in the
proof of Theorem~\ref{thm_IsomProblem}, one can show that $\varphi$ preserves the edge set $E_\omega^0$. It follows that
$\varphi(1)$ equals either $1$ or $-1$ and so $\varphi$ is either the identity map or the inversion $z\mapsto -z$. Hence
$Aut(X_{\omega}) \cong Aut(X_{0^\infty}) \cong \mathbb{Z}/2\mathbb{Z}$.
\end{proof}

Recall that we denoted by $X^n_{\omega}$ the graph with the vertex
set $\mathbb{Z}$ and the edge set $\coprod_{k=0}^nE^k_{\omega}$.
The following proposition solves the isomorphism problem for these
graphs.

\begin{prop}\label{prop_X^n_Iso}
The graphs $X^n_{\omega}$ and $X^n_{\omega'}$ are isomorphic for
all $\omega,\omega'\in \{0,1\}^{\infty}$ and $n\geq 1$. The
automorphism group $Aut(X^n_{\omega})$ is isomorphic to the
infinite dihedral group $D_{\infty}$ for every $n\geq 1$.
\end{prop}
\begin{proof}
As we mentioned when constructing these graphs, the graph
$X_{\omega}^n$ only depends on the first $n$ letters of the
sequence $\omega$. Let $\omega$ and $\omega'$ start with letters
$x_1x_2\ldots x_n$ and $y_1y_2\ldots y_n$, respectively. Then the
translation map $\Phi: V(X_\omega^n) \to V(X_{\omega'}^n)$ defined
by
\[
\Phi(z)=z+(x_1-y_1)+(x_2-y_2)2+\cdots+(x_n-y_n)2^{n-1},
\]
for all $z\in V(X_\omega^n)= \mathbb{Z}$, yields a graph
isomorphism between $X^n_{\omega}$ and $X^n_{\omega'}$ (see
Lemma~\ref{lemma_JustOneLetter}).

As for the corresponding automorphism groups, notice that the
translation $\alpha: z\mapsto z+2^n$ and the inversion $\beta:
z\mapsto -z$ are automorphisms of the graph $X^n_{0^{\infty}}$. As
in the proof of Theorem~\ref{thm_IsomProblem}, one can show that
every automorphism preserves the edge set $E_\omega^0$,  and hence
$Aut(X^n_{0^{\infty}})$ is a subgroup of $Aut(\mathbb{Z})\cong
D_{\infty}$. Since $\alpha$ and $\beta$ generate $D_\infty$, we
deduce that $Aut(X^n_{0^{\infty}})$ is also isomorphic to the
group $D_{\infty}$.
\end{proof}

\section{Dense holonomy and the local isomorphism problem for the graphs $X_{\omega}$}

In this section we consider the local structure of the graphs
$X_{\omega}$, with $\omega\in \{0,1\}^\infty$. We construct a
sequence of finite $4$-regular graphs $X_n = (V_n, E_n)$, for
$n\geq 1$, which will be used to approximate the graphs
$X_{\omega}$. For each $n$, the vertices of the graph $X_n$ are
the residues $\mathbb{Z}/2^n\mathbb{Z}=\{0,1,\ldots,2^n-1\}$
modulo $2^n$, and the edges are $(z \bmod 2^n, z+1 \bmod 2^n)$ and
$$
(2^kz+2^{k-1}\bmod 2^n, 2^k(z+1)+2^{k-1} \bmod 2^n),
$$
for every $z\in\mathbb{Z}$ and $k\geq 1$. Every graph $X_n$ is
$4$-regular with two loops at the vertices $0$ and $2^{n-1}$. The
graphs $X_1$, $X_2$, and $X_3$ are shown in
Figure~\ref{finitegraphsXn}.

\begin{figure}[h]
\begin{center}
\psfrag{0}{$0$}\psfrag{1}{$1$}
\psfrag{00}{$0$}\psfrag{11}{$1$}\psfrag{22}{$2$}\psfrag{33}{$3$}\psfrag{bullet}{$\bullet$}
\psfrag{2}{$2$}\psfrag{3}{$3$}\psfrag{4}{$4$}\psfrag{5}{$5$}\psfrag{6}{$6$}\psfrag{7}{$7$}
\includegraphics[width=0.9\textwidth]{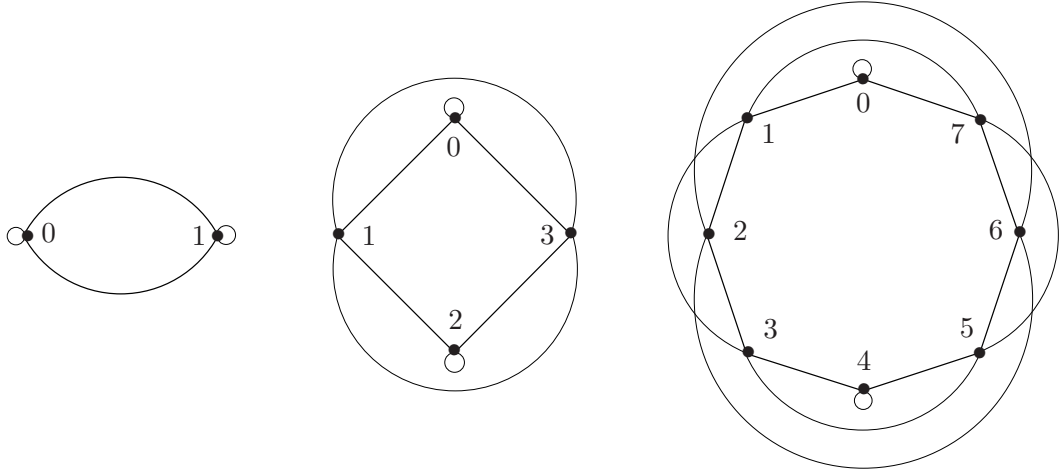} \caption{The graphs $X_1,X_2$ and $X_3$.}\label{finitegraphsXn}
\end{center}
\end{figure}

Take a sequence $\omega\in\{0,1\}^{\infty}$ and consider the graph
$X_{\omega}$. There is a natural sequence of quotients
$X_{\omega}\bmod{2^n}$ of the graph $X_{\omega}$ when we factorize
the vertex set $\mathbb{Z}$ modulo $2^n$. The vertices of the
graph $X_{\omega}\bmod{2^n}$ are the residues modulo $2^n$ and two
vertices $z_1+2^n\mathbb{Z}$ and $z_2+2^n\mathbb{Z}$ are adjacent
if the integers $z_1+2^nt_1$ and $z_2+2^nt_2$ are adjacent in the
graph $X_{\omega}$, for some $t_1,t_2\in\mathbb{Z}$. For example,
the graph $X_{(10)^{\infty}}\bmod{8}$ is shown in
Figure~\ref{10mod8}.

\begin{figure}[h]
\begin{center}
\psfrag{00}{$0$}\psfrag{11}{$1$}\psfrag{22}{$2$}\psfrag{33}{$3$}\psfrag{44}{$4$}\psfrag{55}{$5$}
\psfrag{66}{$6$}\psfrag{77}{$7$}\psfrag{bullet}{$\bullet$}
\includegraphics[width=0.4\textwidth]{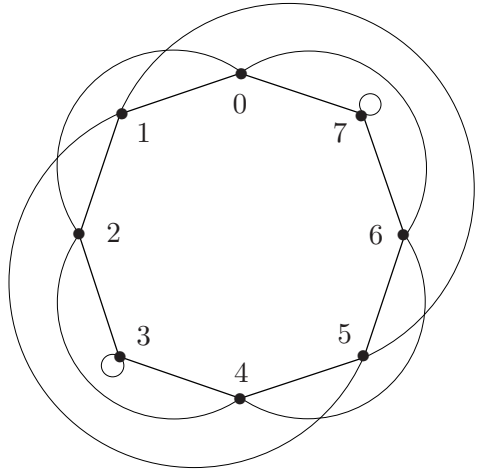} \caption{The graph
$X_{(10)^{\infty}}\bmod{8}$.}\label{10mod8}
\end{center}
\end{figure}

\begin{prop}\label{prop_QuotGraph_Xn}
For every $\omega \in \{0,1\}^\infty$, the quotient graph
$X_{\omega} \bmod 2^n$ is isomorphic to the graph $X_n$.
\end{prop}
\begin{proof}
Let $\omega$ start with the letters $x_1x_2\ldots x_n$, with
$x_i\in \{0,1\}$. Then, as in the proof of
Proposition~\ref{prop_X^n_Iso}, we have that the map $\Phi:
V(X_\omega \bmod{2^n})\to V_n$ defined by
\[
\Phi(z)=z+x_1+x_22+\cdots+x_n2^{n-1} \bmod{2^n},
\]
for all $z\in V(X_\omega \bmod{2^n}) = \mathbb{Z}/2^n\mathbb{Z}$,
yields a graph isomorphism between the graphs $X_{\omega} \bmod
2^n$ and $X_n$.
\end{proof}

Moreover, the graphs $X_{\omega}$ can be recovered as limits of
suitable sequence of pointed graphs $(X_n,v_n)$ in the local
topology. To be more precise, recall that a sequence of pointed
graphs $(Y_n,v_n)$ \textit{converges} to a pointed graph $(Y,v)$
in the \textit{Gromov-Hausdorff metric} if, for every ball
$B_{Y}(v,r)$ in the graph $Y$ with center $v$ and radius $r$,
there exists an isomorphism between $B_{Y}(v,r)$ and
$B_{Y_n}(v_n,r)$ that maps $v$ to $v_n$, for all $n$ large enough.

\begin{thm}
For every sequence $\omega=x_1x_2\ldots\in\{0,1\}^{\infty}$, one
has the convergence
\begin{eqnarray}\label{federico}
(X_{\omega},0)=\lim_{n\rightarrow\infty} (X_n,x_1+x_22+\cdots
+x_n2^{n-1}).
\end{eqnarray}
Moreover, every limit of a convergent sequence of pointed graphs
$(X_n,v_n)$ is isomorphic to a graph $X_{\omega}$, for some
$\omega\in \{0,1\}^{\infty}$.
\end{thm}
\begin{proof}
Let $I_m$ be the subgraph of the graph $X_{\omega}$ induced by the
vertices in the interval $(-2^m,2^m)$. Choose $n$ such that, if a
vertex $v$ of $I_m$ is adjacent to some vertex $u$ of
$X_{\omega}$, then their difference $|v-u|$ is less than $2^n$.
Then the natural projection $I_m\rightarrow I_m\bmod{2^n}$ is an
isomorphism. Composing this projection with the isomorphism from
the proof of Proposition~\ref{prop_QuotGraph_Xn} we get the
isomorphism between $I_m$ and the induced subgraph of $X_n$ that
maps $0$ to $x_1+x_22+\cdots +x_n2^{n-1}$. Since the subgraphs
$\{I_m\}_{m\geq 1}$ cover the graph $X_{\omega}$, we deduce
\eqref{federico}.

For the second statement, consider a convergent sequence of
pointed graphs $(X_n,v_n)$. Using the diagonal argument construct
a sequence $\omega=x_1x_2\ldots\in\{0,1\}$ such that for every $m$
the equality
\[
x_1+x_22+\cdots+x_m2^{m-1}\equiv v_n \bmod{2^m}
\]
holds for infinitely many $n$. By passing to a subsequence, if
necessary, we can assume that this equality holds for all $n$.
Then the sequence $(X_n,v_n)$ converges to the graph
$(X_{\omega},0)$ by the first statement.
\end{proof}

In particular, locally at every point the graph $X_{\omega}$ looks
like a part of some graph $X_n$. This can be used to classify the
graphs $X_{\omega}$ up to local isomorphisms. Two graphs $X$ and
$Y$ are said to be \textit{locally isomorphic} if, for every ball
in one graph, there exists an isomorphic ball in the other graph.

\begin{thm}\label{thm_LocalIsom}
Let $\omega,\omega'\in\{0,1\}^{\infty}$. The graphs $X_{\omega}$
and $X_{\omega'}$ are locally isomorphic if and only if either
both $\omega$ and $\omega'$ are cofinal or anticofinal with
$0^{\infty}$, or both $\omega$ and $\omega'$ are neither cofinal
nor anticofinal with $0^{\infty}$. In particular, the family of
graphs $X_{\omega}$, with $\omega\in\{0,1\}^{\infty}$, contains
precisely two graphs up to local isomorphisms, for example,
$X_{0^{\infty}}$ and $X_{(10)^{\infty}}$.
\end{thm}
\begin{proof}
If $\omega$ and $\omega'$ are cofinal or anticofinal with
$0^{\infty}$, then the graphs $X_{\omega}$ and $X_{\omega'}$ are
even isomorphic by Theorem~\ref{thm_IsomProblem}.

If $\omega$ contains both infinitely many $0$'s and $1$'s, while
$\omega'$ is either cofinal or anticofinal with $0^{\infty}$, then
the graph $X_{\omega'}$ contains a vertex with a loop in contrast
to the graph $X_{\omega}$. Hence, these graphs are not locally
isomorphic.

We are left to consider the case when both $\omega$ and $\omega'$
contain infinitely many $0$'s and $1$'s. Let $\omega=x_1x_2\ldots
$ and $\omega'=y_1y_2\ldots$. Consider the induced subgraph $I_n$
of the graph $X_{\omega}$, whose vertices are the integers in the
interval $(-2^{n-1},2^{n-1})$. Then the map $\Phi$ from the proof
of Proposition \ref{prop_X^n_Iso} is an isomorphism between $I_n$
and the induced subgraph of $X_{\omega'}$, whose vertices are the
integers in the interval $(z_0-2^{n-1},z_0+2^{n-1})$ with
$z_0=(x_1-y_1)+(x_2-y_2)2+\cdots+(x_n-y_n)2^{n-1}$. Since the
subgraphs $\{I_n\}_{n\geq 1}$ eventually cover any ball, the
graphs $X_{\omega}$ and $X_{\omega'}$ are locally isomorphic.
\end{proof}

Let $X$ be a graph of uniformly bounded valence, and consider all
balls $B(v,r)$ of radius $r$ in $X$. We consider the isomorphism
relation on the balls as on pointed graphs, where $B(v,r)$ and
$B(u,r)$ are isomorphic if there exists a graph isomorphism
$B(v,r)\rightarrow B(u,r)$ mapping $v$ to $u$. Note that, since
$X$ has uniformly bounded valence, the set $T_X(r)$ of isomorphism
classes of pointed $r$-balls is finite. We then define the
\textit{$r$-type} of a vertex $v$ of $X$ as the element
$\alpha(v,r) \in T_X(r)$ representing the ball $B(v,r)$. Following
Gromov~\cite{gromov}, we say that the graph $X$ has \emph{dense
holonomy} if for any radius $r$ there exists $R=R(r)$ such that
every $R$-ball in $X$ contains vertices of each $r$-type.
Equivalently, for any radius $r$ there exists $R=R(r)$ such that
for every $r$-type $\alpha\in T_X(r)$ the balls $B(v,R)$ at the
vertices $v$ of type $\alpha(v,r)=\alpha$ cover the whole of the
graph $X$.

\begin{thm}\label{thmholonomy}
The graph $X_{\omega}$, with $\omega\in \{0,1\}^\infty$, has dense
holonomy if and only if the sequence $\omega$ has both infinitely
many $0$'s and $1$'s.
\end{thm}
\begin{proof}
If $\omega$ is cofinal or anticofinal with $0^{\infty}$, then the
graph $X_{\omega}$ contains a unique vertex with loop, and hence
it cannot have dense holonomy.

Suppose that $\omega$ has infinitely many $0$'s and $1$'s. Given
$r\geq 0$, consider the finite set $T_{X_{\omega}}(r)$ and let
$z_1,z_2,\ldots,z_m$ be the representatives for the corresponding
$r$-types. Choose $n$ large enough so that every ball $B(z_i,r+1)$
is contained in the interval $(-2^{n-1},2^{n-1})$. Then the
$r$-type of each vertex $z_i$ coincides with its $r$-type as a
vertex of the graph $X_{\omega}\bmod{2^{n+1}}$. It follows that
$\alpha(z_i+k2^{n+1},r)=\alpha(z_i,r)$ for all $k\in\mathbb{Z}$,
and the property in the definition of dense holonomy holds with
$R=2^n$.
\end{proof}

Finally, we say that a graph $X$ is \emph{generic} if distinct
vertices $u,v$ of $X$ have distinct $r$-types $\alpha(u,r) \neq
\alpha(v,r)$ for some $r$. We observe that genericity is
equivalent to the triviality of the automorphism group. Indeed, if
the balls $B(u,r)$ and $B(v,r)$ are not isomorphic as pointed
graphs, then there is no automorphism $\varphi\in Aut(X)$ that
maps $v$ to $u$. Therefore, if the graph $X$ is generic, then the
automorphism group $Aut(X)$ is trivial. Conversely, suppose that
two distinct vertices $v$ and $u$ have the same $r$-types, for all
$r\geq 0$. Then, for every $r$, there exists a graph isomorphism
$\varphi_r:B(v,r)\to B(u,r)$ that maps $v$ to $u$. Note that the
restrictions $\varphi_r|_{B(v,r')}: B(v,r')\to B(u,r')$, with $r'
\leq r$, are also isomorphisms and that there are only finitely
many graph isomorphisms $B(v,r')\to B(u,r')$. By a compactness
argument, there exists an isomorphism $\varphi\in Aut(X)$ such
that $\varphi(u)=v$, and therefore $Aut(X)$ is nontrivial. Hence,
Corollary~\ref{cor_AutomorphisGroup} implies that the graph
$X_{\omega}$ for $\omega\in\{0,1\}^{\infty}$ is generic if and
only if the sequence $\omega$ has both infinitely many $0$'s and
$1$'s.

\section{The graphs $X_{\omega}$ as Schreier graphs of a self-similar group}

A faithful action of a group $G$ on the space $\{0,1\}^{\infty}$
is called \textit{self-similar} if, for every $g\in G$ and every
$x\in\{0,1\}$, there exists $h\in G$ and $y\in\{0,1\}$ such that
$$
g(x\omega)=yh(\omega),
$$
for every sequence $\omega\in\{0,1\}^{\infty}$. In this case, the
element $h$ is called the \textit{restriction} of $g$ at $x$ and
is denoted by $h=g|_x$. We also get the action of $G$ on
$\{0,1\}$, where $y$ is the image of $x$ under $g$. Then, every
element $g\in G$ can be uniquely given by the tuple
$(g|_0,g|_1)\pi$, where $\pi\in\Sym(\{0,1\})$ is the permutation
induced by $g$ on $\{0,1\}$. Inductively, we can define the action
of $G$ on the set $\{0,1\}^n$, and the restriction
$g|_{x_1x_2\ldots x_n}=(((g|_{x_1})|_{x_2})\ldots)|_{x_n}$, for
any $x_i\in\{0,1\}$.

If a group $G$ acts self-similarly on the space $\{0,1\}^{\infty}$, it can also
be regarded as an automorphism group of the rooted binary tree
$T_2$ (see Figure \ref{tree}). In fact, the $2^n$ vertices of the
$n$-th level of the tree can be identified with the words in
$\{0,1\}^n$, for each $n\geq 1$ (the root of the tree is
identified with the empty word $\emptyset$). The elements of the
boundary $\partial T_2$ of the tree can be identified with the
(right-)infinite binary words, i.e., the elements of
$\{0,1\}^\infty$. In particular, for every automorphism $g\in G$
whose self-similar representation is $g=(g|_0,g|_1)\pi$, the
permutation $\pi\in Sym(\{0,1\})$ describes the action of $g$ on
the first level of the tree, and $g|_i$ is its restriction on the
subtree rooted at the vertex $i$ of the first level, with
$i\in\{0,1\}$. More generally, $g|_{x_1\ldots x_n}$ is the
restriction of the action of $g$ to the subtree rooted at the
vertex $x_1\ldots x_n$ of the $n$-th level of $T_2$. Observe that
such a subtree is isomorphic to the whole tree $T_2$. Then, the
property of self-similarity means that these restrictions are
elements of $G$.

\begin{figure}[h]
\begin{picture}(500,250)
\put(250,240){\circle*{2}}
\put(170,180){\circle*{2}}\put(330,180){\circle*{2}}
\put(130,120){\circle*{2}}\put(210,120){\circle*{2}}\put(290,120){\circle*{2}}
\put(370,120){\circle*{2}}

\put(110,60){\circle*{2}}\put(150,60){\circle*{2}}
\put(190,60){\circle*{2}}\put(230,60){\circle*{2}}\put(270,60){\circle*{2}}
\put(310,60){\circle*{2}}\put(350,60){\circle*{2}}\put(390,60){\circle*{2}}

\put(250,240){\line(-4,-3){80}}\put(250,240){\line(4,-3){80}}
\put(170,180){\line(-2,-3){40}}\put(170,180){\line(2,-3){40}}
\put(330,180){\line(-2,-3){40}}\put(330,180){\line(2,-3){40}}

\put(104,40){\circle*{1}} \put(116,40){\circle*{1}}
\put(144,40){\circle*{1}} \put(156,40){\circle*{1}}
\put(184,40){\circle*{1}} \put(196,40){\circle*{1}}
\put(224,40){\circle*{1}} \put(236,40){\circle*{1}}
\put(264,40){\circle*{1}} \put(276,40){\circle*{1}}
\put(304,40){\circle*{1}} \put(316,40){\circle*{1}}
\put(344,40){\circle*{1}} \put(356,40){\circle*{1}}
\put(384,40){\circle*{1}} \put(396,40){\circle*{1}}

\put(104,30){\circle*{1}}\put(116,30){\circle*{1}}
\put(144,30){\circle*{1}} \put(156,30){\circle*{1}}
\put(184,30){\circle*{1}} \put(196,30){\circle*{1}}
\put(224,30){\circle*{1}} \put(236,30){\circle*{1}}
\put(264,30){\circle*{1}} \put(276,30){\circle*{1}}
\put(304,30){\circle*{1}} \put(316,30){\circle*{1}}
\put(344,30){\circle*{1}} \put(356,30){\circle*{1}}
\put(384,30){\circle*{1}}\put(396,30){\circle*{1}}

\put(104,20){\circle*{1}} \put(116,20){\circle*{1}}
\put(144,20){\circle*{1}}
\put(156,20){\circle*{1}}\put(184,20){\circle*{1}}
\put(196,20){\circle*{1}} \put(224,20){\circle*{1}}
\put(236,20){\circle*{1}} \put(264,20){\circle*{1}}
\put(276,20){\circle*{1}} \put(304,20){\circle*{1}}
\put(316,20){\circle*{1}} \put(344,20){\circle*{1}}
\put(356,20){\circle*{1}} \put(384,20){\circle*{1}}
\put(396,20){\circle*{1}}

\put(247,243){$\emptyset$} \put(159,180){$0$}\put(333,180){$1$}
\put(115,122){$00$}\put(210,122){$01$}
\put(273,122){$10$}\put(372,122){$11$}
\put(104,48){$000$}\put(139,48){$001$}
\put(179,48){$010$}\put(219,48){$011$}
\put(259,48){$100$}\put(299,48){$101$}
\put(339,48){$110$}\put(379,48){$111$}

\put(130,120){\line(-1,-3){20}}\put(130,120){\line(1,-3){20}}
\put(210,120){\line(-1,-3){20}}\put(210,120){\line(1,-3){20}}
\put(290,120){\line(-1,-3){20}}\put(290,120){\line(1,-3){20}}
\put(370,120){\line(-1,-3){20}}\put(370,120){\line(1,-3){20}}
\end{picture}
\caption{The rooted binary tree $T_2$.}\label{tree}
\end{figure}
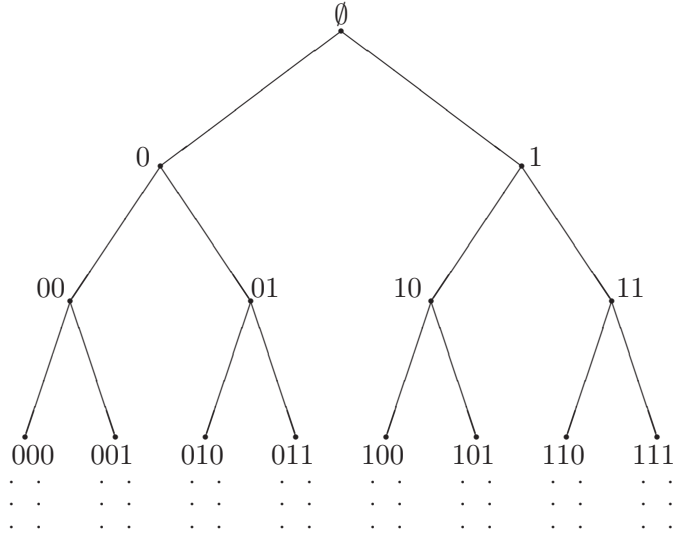

Consider the self-similar group $G$ generated by the
transformations $a$ and $b$ of the set $\{0,1\}^{\infty}$, whose
actions satisfy the following recursive rules
\[
\begin{array}{rclrcl}
a(0x_2x_3\ldots)&=&1x_2x_3\ldots,\qquad\qquad  &b(0x_2x_3\ldots)&=&0b(x_2x_3\ldots),\\
a(1x_2x_3\ldots)&=&0a(x_2x_3\ldots),\qquad\qquad
&b(1x_2x_3\ldots)&=&1a(x_2x_3\ldots),
\end{array}
\]
for all $x_i\in\{0,1\}$. Using restrictions, the generators $a$
and $b$ can be written recursively as
\[
a=(e,a)\sigma,\qquad\qquad b=(b,a),
\]
where $\sigma$ is the transposition $(0,1)$, and $e$ is the
identity transformation.

The group $G$ is the simplest example of a group generated by a
polynomial but not bounded automaton (see definition in
\cite{sidki:circ} or \cite[Chapter~IV]{PhDBondarenko}, the
generating automaton of $G$ is shown in
Figure~\ref{fig_Automaton}).

\begin{figure}[h]
\begin{center}
\psfrag{b}{$b$} \psfrag{a}{$a$} \psfrag{id}{$id$}
\psfrag{1}{$0|0$} \psfrag{2}{$1|1$} \psfrag{3}{$1|0$}
\psfrag{4}{$0|1$} \psfrag{5}{$0|0$} \psfrag{6}{$1|1$}
\epsfig{file=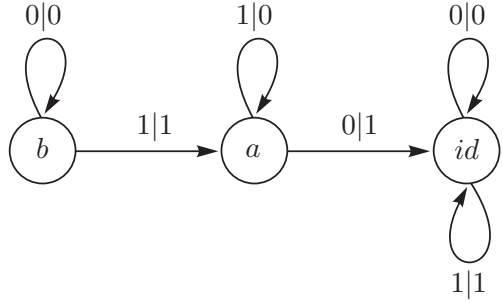,height=110pt} \caption{The
generating automaton for the group $G$.}\label{fig_Automaton}
\end{center}
\end{figure}

Moreover, all the groups generated by polynomial and not bounded
automata minimizing the sum of the number of states and the number
of letters are isomorphic to the group $G$. Indeed, all such
generating automata have $3$ states over an alphabet with $2$
letters, and one state defines the trivial automorphism. Up to
passing to inverses of generators, permuting the states of the
automaton, permuting letters of the alphabet, there are only $3$
such automata, whose states satisfy the following recursions:
\[\begin{array}{lll}
a=(e,a)\sigma\qquad &a_1=(e,a_1)\sigma\qquad &a_2=(e,a_2)\sigma\\
b=(b,a)\qquad   &b_1=(b_1,a_1)\sigma\qquad &b_2=(a_1,b_2)\sigma.
\end{array}\]
Then $a=a_1=a_2$, $b_1a=(b_1a,a)$, $a^{-1}b_2=(a^{-1}b_2,a)$. The
transformations $b$, $b_1a$, and $a^{-1}b_2$ satisfy the same
recursion and thus are all equal. Hence $\{a,b\}$, $\{a_1,b_1\}$,
and $\{a_2,b_2\}$ are just different generating sets of the group
$G$.


The action of the transformation $a=(e,a)\sigma$ on the space
$\{0,1\}^{\infty}$ corresponds to the addition of $1$ to dyadic
integers $\mathbb{Z}_2$, when the sequence $x_1x_2x_3\ldots$ is
identified with the binary integer $x_1+x_22+x_32^2+\ldots$; for
this reason, $a$ is called the (binary) \textit{adding machine}.
Indeed, $a(x_1x_2\ldots )=y_1y_2\ldots $, with $x_i,y_i\in
\{0,1\}$, if and only if
\[
1+x_1+x_22+\ldots+x_n2^{n-1}+\ldots=y_1+y_22+\ldots+y_n2^{n-1}+\ldots.
\]
In particular, the action of $a$ is transitive on every set
$\{0,1\}^n$, and hence the group $G$ also acts transitively on
$\{0,1\}^n$, for every $n\geq 1$. This is expressed by saying that
the action of $G$ is {\it level-transitive} on the tree. In the
Appendix, it is also shown that $G$ is self-replicating
(recurrent) and regular weakly branch over its commutator
subgroup.

\begin{prop}\label{proposition_orbits}
The union of the cofinality classes $\Cof (0^{\infty})\cup \Cof
(1^{\infty})$ forms one orbit of the action of $G$ on
$\{0,1\}^{\infty}$. Any other orbit consists of precisely one
cofinality class.
\end{prop}
\begin{proof}
The orbits of the action of the adding machine $a$ on the ring
$\mathbb{Z}_2$ of dyadic integers precisely correspond to the
cofinality classes in the statement. Hence, it is sufficient to
prove that the generator $b$ preserves these orbits. To achieve
this, notice that if $b(x_1x_2\ldots)=y_1y_2\ldots$ and we take
the first position $n$ with $x_n\neq y_n$, then $b|_{x_1x_2\ldots
x_{n-1}}=a$ and $a(x_nx_{n+1}\ldots)=y_ny_{n+1}\ldots$.
\end{proof}

Every action of a finitely generated group can be described by the
associated Schreier graph. The \textit{Schreier graph $\gr_n$} of
the action of the group $G$ on the set $\{0,1\}^n$ is the graph
with the set of vertices $\{0,1\}^n$ and there is an edge between
$v$ and $s(v)$ labeled by $s$, for every $v\in \{0,1\}^n$ and
$s\in \{a,b\}$. For a point $\omega\in\{0,1\}^{\infty}$ the
\textit{(orbital) Schreier graph $\gr_{\omega}$} of $\omega$ under
the action of $G$ is the graph, whose vertex set is the orbit
$G(\omega)$ of $\omega$ under the action of $G$, and there is an
edge between every two vertices $v$ and $s(v)$ labeled by $s$ for
every $s\in\{a,b\}$. It follows from Proposition
\ref{proposition_orbits} that the vertex set of the graph
$\gr_{0^{\infty}}$ is the union
$\Cof(0^{\infty})\cup\Cof(1^{\infty})$; every other graph
$\gr_{\omega}$ has the set of vertices $\Cof(\omega)$. All the
Schreier graphs $\gr_n$ and $\gr_{\omega}$ are connected. For
every $\omega=x_1x_2\ldots\in\{0,1\}^{\infty}$, the sequence of
the pointed Schreier graphs $(\gr_n,x_1x_2\ldots x_n)$ converges
in the local topology on pointed graphs to the pointed graph
$(\gr_{\omega},\omega)$ (see
\cite[Proposition~7.2]{fractal_gr_sets} and
\cite{gri_zuk:asympt_spect}). Schreier graphs of self-similar
actions of groups have been largely studied from the viewpoint of
spectral computations, growth, amenability, topology of Julia sets
\cite{barth, PhDBondarenko, Bondarenkoarxiv, grigor-nekra, Geneva,
grigo-sunik2, grigo-sunik1}.

It is shown in \cite{omega_periodic} that the graph
$X_{0^{\infty}}$ is isomorphic to the Schreier graph of the
infinite word $0^{\infty}$ under the action of the group $G$. We
generalize this analysis by the following statement.

\begin{thm}\label{thmschreierisom}
The map $\varphi:V(\gr_{\omega})\rightarrow V_{\omega}$ defined by
the rule
\[
\varphi(a^m(\omega))=m \quad \mbox{ for } m\in\mathbb{Z}
\]
is a graph isomorphism for every $\omega\in\{0,1\}^{\infty}$.

The map $\varphi_n:V(\gr_n)\rightarrow V_n$ defined by the rule
\[
\varphi_n(a^m(0^n))=m \quad \mbox{ for } m=0,1,\ldots, 2^n-1,
\]
is a graph isomorphism for every $n\geq 1$.
\end{thm}
\begin{proof}
We will only prove the first statement, the second is similar.

The map $\varphi$ is well defined and bijective because the action
of $a$ is faithful and transitive on every orbit. We need to show
that $\varphi$ preserves the adjacency relation. First, if two
vertices of the graph $\gr_{\omega}$ are joined by an edge labeled
by $a$, then the corresponding binary words differ by $\pm 1$ and
so they are mapped to vertices of the graph $X_{\omega}$ which are
connected by an edge of $E^0_\omega$. We are left to consider
edges labeled by $b$. The actions of $a$ and $b$ have the
following properties:
\[
b(0^{n-1}1\omega)=0^{n-1}1a(\omega) \quad \mbox{ and } \quad
a^{2^n}(v\omega)=va(\omega),
\]
for all $v\in \{0,1\}^n$ and $\omega\in\{0,1\}^{\infty}$. Let
$\omega=x_1x_2\ldots x_n\omega'$, with $x_i\in\{0,1\}$. Using the
correspondence with binary numbers and the definition of
$a_n^{\omega}$ we get
\[
x_1+x_22+\cdots +x_n2^{n-1}+\omega'2^n -
a_n^{\omega}=2^{n-1}+\omega'2^n \quad \Rightarrow \quad
a^{-a_n^{\omega}}(\omega)=0^{n-1}1\omega'.
\]
Hence
\[
b(a^{2^nk-a_n^{\omega}}(\omega))=b(0^{n-1}1a^{k}(\omega'))=0^{n-1}1a^{k+1}(\omega')=a^{2^n(k+1)-a_n^{\omega}}(\omega),
\]
which corresponds to the edge
$(2^nk-a_n^{\omega},2^n(k+1)-a_n^{\omega})$ of the graph
$X_{\omega}$.
\end{proof}

We can apply the results from the previous section to the graphs
$\gr_{\omega}$.

\begin{cor}
Let $\omega,\omega'\in \{0,1\}^{\infty}$. The orbital Schreier
graphs $\gr_{\omega}$ and $\gr_{\omega'}$ are isomorphic if and
only if $\omega$ and $\omega'$ are cofinal or anticofinal.
\end{cor}

If $\omega$ is eventually constant (and so cofinal or anticofinal
with $0^{\infty}$), then $\gr_{\omega}$ and $\gr_{\omega'}$ are
isomorphic if and only if also $\omega'$ is eventually constant
and in this case $\gr_{\omega}=\gr_{\omega'}$, because $\omega$
and $\omega'$ belong to the same orbit. If $\omega$ contains
infinitely many $0$'s and infinitely many $1$'s, then
$\gr_{\omega}=\gr_{\omega'}$ if $\omega'$ is cofinal with $\omega$
(they are in the same orbit), and $\gr_{\omega}\cong\gr_{\omega'}$
if $\omega'$ is anticofinal with $\omega$ (they are in different
orbits).

\begin{cor}
The set $\{\gr_{\omega}\}_{\omega\in \{0,1\}^{\infty}}$ contains
uncountably many isomorphism classes of graphs. The class of
$\gr_{0^{\infty}}$ contains only one graph; any other class
contains two isomorphic graphs (corresponding to two different
orbits).

Moreover, all graphs $\gr_{\omega}$ except $\gr_{0^{\infty}}$ are
locally isomorphic.
\end{cor}

Let us list at the end of this section known properties of the group $G$.
First of all, we observe that the monoid generated by $a$ and $b$
is free, and hence $G$ has exponential growth (see
\cite{classification} for the proof of these properties, where the
group $G$ appears under the number $929$).

It is proved by S.~Sidki in~\cite{sidki:nonfree} that groups generated by polynomially growing automata do not have free non-abelian subgroups, which implies that $G$ has no free non-abelian subgroups. For a shorter proof of S.~Sidki's result (for the case of locally finite trees) see \cite{nekrashevych}.

In \cite{amen:linear} it is shown that the class of
linear-activity automata groups is contained in the class of
amenable groups. In particular, the group $G$, which is discussed
there in the example following (the statement of) Theorem 1 (where
the generator $a$ (resp. $b$) corresponds to our generator $b$
(resp. $a^{-1}$)) and which is called the {\it long-range group},
is amenable. This answered a question posed by the fourth named
author in the Kourovka notebook \cite{kourovka}, Question 16.74,
and also in Guido's book of conjectures \cite{guido}, Conjecture
35.9.

Note that the amenability of $G$ provides another proof of the
fact that $G$ contains no free nonabelian subgroups.

\section{Growth of the graphs $X_{\omega}$}

Let $X$ be a connected graph of uniformly bounded valence.
The \textit{growth function} of $X$ with respect to one of its vertices $v \in V(X)$
is defined by $\gamma_v(n)=|B(v,n)|$. In order to avoid dependence
on a vertex, one introduces the equivalence relation on the growth
functions. Given two functions
$f,g:\mathbb{N}\cup\{0\}\rightarrow\mathbb{N}\cup\{0\}$ we say
that $f\prec g$ if there exists a constant $C>0$ and an integer
$n_0$ such that $f(n)\leq g(Cn)$ for all $n\geq n_0$, and then $f$
and $g$ are called equivalent $f\sim g$ if $f\prec g$ and $g\prec
f$. The equivalence class of a function is called its
\textit{growth}. Then, for any two vertices of the graph $X$, the
respective growth functions are equivalent, and one can talk about
the growth $\gamma$ of $X$. We say that the graph $X$ has
\textit{intermediate growth} if its growth $\gamma$ satisfies
$P\precnsim \gamma\precnsim E$ for any polynomial $P$ and every
exponential function $E(n)=a^n$, with $a>1$. Recall that the {\it
diameter} of a finite graph $\Gamma = (V,E)$ is defined as
$\max_{u,v\in V}d(u,v)$.

\begin{thm}\label{thm_growth_diam}
The diameters of the Schreier graphs $\gr_n$ have intermediate
growth. Indeed, there exist constants $c,d>0$ such that
\[
c\sqrt{n}\,2^{\sqrt{2n}}\leq {\rm Diam}(\gr_n)\leq
d\sqrt{n}\,2^{\sqrt{2n}}
\]
for all $n\geq 1$.
\end{thm}
\begin{proof}
The statement basically follows from Lemmas 3 and 4 in
\cite{omega_periodic}. Using the symmetries of the graph $\gr_n$
one can check that the diameter ${\rm Diam}(\gr_n)$ is bounded
from above by the distance $2d(0^n,0^{n-1}1)$, and of course
$d(0^n,0^{n-1}1)$ is the lower bound. It is shown in Lemma~3 that
\[
d\left(0^{\frac{n^2+3n+2}{2}},0^{\frac{n^2+3n}{2}}1\right)=n2^n+1,
\]
for all $n\geq 1$. Hence, for $m\approx \frac{n^2}{2}$, we get that ${\rm Diam}(\gr_m)$ is equal to $\sqrt{m}\,
2^{\sqrt{2m}}$ up to a bounded multiplicative constant dependent on $m$.
\end{proof}

It is proved in \cite{omega_periodic} that the graph
$\gr_{0^{\infty}}$ has intermediate growth. We generalize this
result in the next theorem.

\begin{thm}\label{thmintermediate}
All orbital Schreier graphs $\gr_{\omega}$ for $\omega\in\{0,1\}^{\infty}$ of the group $G$ have intermediate growth. 
\end{thm}
\begin{proof}
The lower bound follows from Theorem~\ref{thm_growth_diam}.
Indeed, every ball $B(\omega,r)$ in the graph $\gr_{\omega}$ of
radius $r\geq 2{\rm Diam}(\gr_n)$ contains at least $2^n$
vertices, which correspond to the integers in the interval
$[0,2^n]$ when we use the identification
$\gr_{\omega}=X_{\omega}$. Hence, the ball $B(\omega,n)$ contains
$\succeq n^{\frac{1}{2}\log_2 n}$ vertices, which gives the
super-polynomial lower bound.

Let us prove the bound from above. Let $l(g)$ for $g\in G$ be the
length of the element $g$ in the generators $a,b$, i.e., $l(g)$ is
equal to the minimal number $n$ such that $g$ can be expressed as
a product $g=s_1s_2\ldots s_n$, with $s_i\in\{a^{\pm 1},b^{\pm
1}\}$. Notice that $l(g|_v)\leq l(g)$ for all words $v$ over the
alphabet $\{0,1\}$.

Fix a sequence $\omega=x_1x_2\ldots\in\{0,1\}^{\infty}$ and
consider the ball $B(\omega,n)$ in the graph $\gr_{\omega}$,
centered at the vertex $\omega$ and of radius $n$. If
$g(v_1\omega_1)=v_2\omega_2$ for $v_1,v_2\in \{0,1\}^k$ and
$\omega_1,\omega_2\in\{0,1\}^{\infty}$ then
$g|_{v_1}(\omega_1)=\omega_2$. Hence, for every fixed $k$, each
sequence in the ball $B(\omega,n)$ is of the form $v\omega_1$ for
some $v\in \{0,1\}^k$ and $\omega_1=h(x_{k+1}x_{k+2}\ldots)$ for
some $h\in \nucl(n,k)$, where
\[
\nucl(n,k)=\{g|_{x_1x_2\ldots x_k} : g\in G \mbox{ and } l(g)\leq
n \}.
\]
It follows that
\begin{equation}\label{eqn_thm_growth_bound_B(w,n)}
|B(\omega,n)|\leq 2^k\cdot |\nucl(n,k)|.
\end{equation}
Let us show that, for every $n$, we can find $k$ such that the
values $2^k$ and $|\nucl(n,k)|$ are small enough.

Consider an element $g\in G$ of length $\leq n$ written as a word
in generators
\begin{equation}\label{eqn_thm_growth_g_generators}
g=a^{\alpha_0} b^{\beta_1} a^{\alpha_1}b^{\beta_2}\ldots
b^{\beta_m}a^{\alpha_m},
\end{equation}
where the interior powers are non-zero. Notice that $a^l|_v\in
\{1,a,a^{-1}\}$ for all words $v$ of length $\geq \log_2 l$, and
$b^l|_v$ is equal to $b^l$ or $a^i$ with $|i|\leq |l|$ for all
finite words $v$. Hence the restriction $g|_v$ for words $v$ of
length $\geq \log_2 n$ can be written in the form
\begin{equation}\label{eqn_thm_growth_b_positions}
a^{\varepsilon_0}\mbox{ {\tt\symbol{"20}} }a^{\varepsilon_1}\mbox{
{\tt\symbol{"20}} }\ldots \mbox{ {\tt\symbol{"20}}
}a^{\varepsilon_m},
\end{equation}
where $\varepsilon_i\in\{-1,0,1\}$ and on every position
{\tt\symbol{"20}} we get a power of $a$ or of $b$. If one of the
places {\tt\symbol{"20}} in the expression  {\tt\symbol{"20}}
$a^{\varepsilon_i}$ {\tt\symbol{"20}} from
(\ref{eqn_thm_growth_b_positions}) is filled with a power of $a$,
then this expression contains at most one position with a power of
$b$. The same holds if $\varepsilon_i=0$. If
$\varepsilon_i\in\{-1,1\}$, then
\begin{align*}
b^{\beta_i}ab^{\beta_{i+1}}|_0&=b^{\beta_i}a^{\beta_{i+1}} &
b^{\beta_i}ab^{\beta_{i+1}}|_1&=a^{\beta_i}ab^{\beta_{i+1}} \\
b^{\beta_i}a^{-1}b^{\beta_{i+1}}|_0&=b^{\beta_i}a^{-1}a^{\beta_{i+1}}
& b^{\beta_i}a^{-1}b^{\beta_{i+1}}|_1&=a^{\beta_i}b^{\beta_{i+1}}
\end{align*}
and in all cases there is only one position with a power of $b$.
Hence, the restriction $g|_v$ for words $v$ of length $\geq
(\log_2 n + 1)$ can be expressed in the form
(\ref{eqn_thm_growth_g_generators}) with $\leq m/2$ positions with
a power of $b$. By applying the same procedure $\log_2 m$ times,
we get an element with at most one position with a power of $b$.
It follows that
\[
g|_v\in \nucl=\{a^{\varepsilon_1}b^ka^{\varepsilon_2}, a^k  :
k\in\mathbb{Z} \mbox{ and } \varepsilon_i\in\{-1,0,1\}\}
\]
for words $v$ of length $\geq (\log_2 n)(\log_2 n+1)$ (here we use
$m\leq n/2$). Notice that the set $\nucl$ contains $\leq 20n$
elements of length $\leq n$. We can apply estimate
(\ref{eqn_thm_growth_bound_B(w,n)}) with $k=(\log_2 n)(\log_2
n+1)$ and get
\[
|B(\omega,n)|\leq 20 n^{\log_2 n + 2}\sim n^{\log_2 n}.
\]
\end{proof}

\bibliographystyle{alpha}
\def\cprime{$'$}

\end{document}